\documentclass[review]{elsarticle}

\usepackage{CJK}
\usepackage{lineno,hyperref}
\usepackage{epsfig}
\usepackage{epstopdf}
\usepackage{inputenc}

\journal{Journal of \LaTeX\ Templates}

\begin{document}

\newtheorem{lemma}{\textbf{Lemma}}
\newtheorem{definition}{\textbf{Definition}}
\newtheorem{theorem}{\textbf{Theorem}}
\newtheorem{example}{\textbf{Example}}
\def\proof{\noindent{\bf Proof}: }

\begin{frontmatter}

\title{Stability Analysis of Time-varying Delay Neural Network for Convex Quadratic Programming With Equality Constraints and Inequality Constraints\tnoteref{mytitlenote}}
\tnotetext[mytitlenote]{This project was supported by Natural Science Foundation of Shangdong Province, No.ZR2019PA007}

\author{Ling Zhang}


\author{Xiaoqi Sun\corref{mycorrespondingauthor}}
\cortext[mycorrespondingauthor]{Corresponding author}
\ead{sunxiaoqi@live.com}

\address{Mathematics and Statistics Department, Qingdao University, Qingdao, China}

\begin{abstract}
In this paper, a kind of neural network with time-varying delays is proposed to solve the problems of quadratic programming. The delay term of the neural network changes with time $t$. The number of neurons in the neural network is $n+h$, so the structure is more concise. The equilibrium point of the neural network is consistent with the optimal solution of the original optimization problem. The existence and uniqueness of the equilibrium point of the neural network are proved. Application inequality technique proved global exponential stability of the network. Some numerical examples are given to show that the proposed neural network model has good performance for solving optimization problems.
\end{abstract}


\end{frontmatter}


\section{Introdution}

\par \setlength{\parindent}{2em} The model of the convex quadratic programming (CQP) problem is simple in form, convenient to construct, and easy to solve, it is now the basic method of learning risk assessment management\cite{article1,article2,article3}, system analysis\cite{article4,article5}, combinatorial optimization science\cite{article6}, economic dispatch\cite{article7}, and other disciplines. The quadratic programming problem is widely used in the fields of robust control\cite{article8}, parameter estimation\cite{article9,article10}, regression analysis\cite{article11}, image and signal processing\cite{article12}, etc. The general form of the CQP model is given below:
\begin{equation}\label{1}
\begin{array}{l}
\min \frac{1}2 x^{T}Qx + c^{T}x := f(x)\\
\left\{
        \begin{array}{l}
\textit{s.t.} Ax=b\\
             Bx\leq d
        \end{array}
\right.
\end{array}
\end{equation}
Where $Q\in R^{n\times n}$ is semi-definite matrix, $x=(x_{1},x_{2},...,x_{n})\in R^{n}$, $c\in R^{n}$, $A\in R^{n\times n}$ is a row full rank matrix, that is $Rank(A)=m$, $b\in R^{m}$, $B\in R^{h\times n}$, $d\in R^{h}$.
\par \setlength{\parindent}{2em} Quadratic programming problem is widely used in practical problems. Due to a large number of dimensions and complex structure in practical problems, calculations with traditional numerical methods will take too long. Solving the quadratic programming problems through artificial neural networks can shorten the calculation time. In 1986, D. Tank and J.J. Hopfield\cite{article13} first proposed to solve the optimization problem by constructing a neural network. After that, the application of neural networks based on circuit implementation to solve quadratic programming problems has become a research topic. Kennedy and Chua\cite{article14} proposed a penalty function to construct neural networks. Using this neural network, the results of Tank and Hopfield are extended to general nonlinear programming problems. Based on the Lagrange multiplier method instead of using penalty functions, Zhang and Constantinides\cite{article15} make the network contain two types of neurons, reducing the restrictions on the form of the cost function. Huang improved on Zhang and Constantinides and designed a new Lagrangian-type neural network, which can directly deal with inequality constraints without adding relaxation variables\cite{article16}. Based on the gradient method, Chen et al. established a type of neural network that did not involve penalty parameters and could solve both the original problem and the dual problem at the same time\cite{article17}.Based on the projection theorem and KKT condition, Xia et al. established a recurrent neural network to solve the related linear piecewise equation, which reduced the complexity of the model\cite{article18}.A type of recursive neural network was proposed by Nazemi et al \cite{article21} to solve quadratic programming problems. The constructed neural network does not need the multiplier related to inequality in quadratic programming conditions.
\par \setlength{\parindent}{2em} In reality, due to the influence of hardware performance and the limitation of signal transmission, it is inevitable to produce time delay in the limited transmission time. The impact of time delay on the operating system is often not negligible. Some papers have proposed neural networks with a time delay to solve quadratic optimization problems. For example, Liu et al. put forward a kind of time-delay neural network to solve the linear projection equation and proved the stability of the time-delay neural network by linear matrix inequality method(LMI) methods. Yang and Cao \cite{article20} proposed a time-delay projection neural network without penalty function and Lagrange multipliers. Its structure is simple and the network state variables are reduced, but this will affect its practicability. Sha et al. \cite{article22} proposed a type of delayed neural network added time delay, with fewer neurons to improve computational efficiency, reduce the number of network structure layers. Wen et al. \cite{article25} generalized the neural network that solves the convex optimization problem, and gave a time-delay neural network to solve the general optimization problem with weak convexity.
\par \setlength{\parindent}{2em}Indeed, the delayed neural networks not contained in a constant value, it will produce a change over time. Therefore, it is meaningful to study neural networks with variable time delays for solving quadratic programming problems. In this paper, a neural network model with variable time delay is established based on the saddle point theorem and the projection theorem. The existence and uniqueness of the network are analyzed, and the global exponential stability of the network is proved by using the inequality technique. Some numerical examples are listed and verified by Matlab2016a. The results show that the neural network has good performance.
\par \setlength{\parindent}{2em}The order of this article is as follows: in the second part, we derive the neural network model with $n+h$ neurons utilizing saddle point theorem, projection theorem, and some inequalities; in the third part, using techniques such as scaling of inequalities, combined with the lemma, we discussed the existence and uniqueness of the proposed neural network with variable delays; in the fourth part, we discuss that the proposed neural network is globally exponentially stable when the condition $(|\kappa-1|+1)\|I-\alpha W\|-\kappa<0$ is satisfied. In the fifth part, the examples of 3 -dimensional and 4 - dimensional convex quadratic programming are given to verify the better performance of the neural networks.

\section{Establishment of neural network model }

\par \setlength{\parindent}{2em}Let $\Omega=\{x\in R^{n}|Ax=b,Bx\leq d\}$ be a non-empty feasible region of $(1)$, and the optimal solution of $(1)$ is in $\Omega$.
The Lagrange function of $(1)$ can be written as
\begin{equation}\label{2}
G(x,u,v)=f(x)-u^{\mathrm{T}}(Ax-b)-v^{\mathrm{T}}(d-Bx)
\end{equation}
where $u\in R^{m},v\in R^{h}$ are Lagrange multipliers.\\
Using the saddle point theorem, if $x^{*}$ is used to represent the optimal solution of $(1)$, then there exists $u^{*},v^{*}$ such that the following inequality holds
\begin{equation}\label{3}
G(x^{*},u,v^{*})\leq G(x^{*},u^{*},v^{*})\leq G(x,u^{*},v^{*})
\end{equation}
We put $(2)$ into $(3)$
\begin{equation}\label{4}
\begin{array}{l}
f(x^{*})-u^{\mathrm{T}}(Ax^{*}-b)-(v^{*})^{\mathrm{T}}(d-Bx^{*})\\ \begin{array}{l}
 \leq f(x^{*})-(u^{*})^{\mathrm{T}}(Ax^{*}-b)-(v^{*})^{\mathrm{T}}(d-Bx^{*})\\ \leq f(x)-(u^{*})^{\mathrm{T}}(Ax-b)-(v^{*})^{\mathrm{T}}(d-Bx)
\end{array}
\end{array}
\end{equation}
It can be obtained from the left side of $(4)$ $(-u^{T}+(u^{*})^{\mathrm{T}})(Ax^{*}-b)\leq 0$
that is
\begin{equation}\label{5}
(u-(u^{*}))^{\mathrm{T}}(Ax^{*}-b)\geq 0,\forall u\in R^{m}
\end{equation}
so as to get $Ax^{*}=b$.\\
It can be obtained from the right side of $(4)$ $f(x^{*})-(u^{*})^{\mathrm{T}}Ax^{*}-f(x)+(u^{*})^{\mathrm{T}}Ax\leq v^{\mathrm{T}}Bx-(v^{*})^{\mathrm{T}}Bx^{*}$
that is\begin{equation}\label{6}
f(x^{*})-(u^{*})^{\mathrm{T}}Ax^{*}+c^{*}-(f(x)-(u^{*})^{\mathrm{T}}Ax+(v^{*})^{\mathrm{T}}Bx)\leq 0
\end{equation}
It was found that $x^{*}=\min\{f(x)-(u^{*})^{\mathrm{T}}Ax+(v^{*})^{\mathrm{T}}Bx\}$. Thus, $x^{*}, u^{*}, v^{*}$ satisfy
$\nabla f(x^{*})-A^{T}u^{*}+B^{T}v^{*}=0$, $\nabla f(x^{*})$ denotes the gradient of the differential function $f(x)$, that is
\begin{equation}\label{7}
x^{*}=Q^{-1}(A^{T}u^{*}-B^{T}v^{*}-c).
\end{equation}\\
From $(7)$, and we already know $Ax^{*}=b$, it can be deduced that $Ax^{*}-A(Qx^{*}+c)+AA^{\mathrm{T}}u^{*}-AB^{\mathrm{T}}v^{*}=b$, that is $AA^{\mathrm{T}}u^{*}=b-Ax^{*}+A(Qx^{*}+c)+AB^{\mathrm{T}}v^{*}=A(Qx^{*}+c+B^{\mathrm{T}}v^{*})-(Ax^{*}-b)$,
thus
\begin{equation}\label{8}
Au^{*}=A^{\mathrm{T}}(AA^{\mathrm{T}})^{-1}A(Qx^{*}+c+B^{\mathrm{T}}v^{*})-A^{\mathrm{T}}(AA^{\mathrm{T}})^{-1}(Ax^{*}-b)
\end{equation}
and we have
\begin{equation}\label{9}
A^{\mathrm{T}}u^{*}=Qx^{*}+c+B^{\mathrm{T}}v^{*}
\end{equation}
so that $(\mathrm{I}-A^{\mathrm{T}}(AA^{\mathrm{T}})^{-1}A)(Qx^{*}+c+B^{\mathrm{T}}v^{*})+A^{\mathrm{T}}(AA^{\mathrm{T}})^{-1}(Ax^{*}-b)=0$.
Now write
\begin{equation}\label{10}
M:=A^{\mathrm{T}}(AA^{\mathrm{T}})^{-1}A\in R^{n\times n},N:=A^{\mathrm{T}}(AA^{\mathrm{T}})^{-1}\in R^{n\times m}
\end{equation}
we have
\begin{equation}\label{11}
(\mathrm{I}-M)(Qx^{*}+c+B^{\mathrm{T}}v^{*})+N(Ax^{*}-b)=0.
\end{equation}
By the projection theorem, $(v+\alpha(Bx-d))^{+}-v=0$, that is
\begin{equation}\label{12}
v^{*}=P_{\Omega}(v^{*}+\alpha(Bx-d))
\end{equation}
where $\alpha\in R^{+}$,$R^{+}={s|s>0,s\in R}$,$(v)^{+}=(v_{1},v_{2},...,v_{h})\in R^{h}$ and $(v_{i})^{+}=\max [v_{i},0]$.\\
Through linear transformation, we can deduce $\exists \gamma$ to make $(8)$ into the following formula
\begin{equation}\label{13}
(v+\alpha((Bx-d)-\gamma[(\mathrm{I}-M)(Qx^{*}+c+B^{\mathrm{T}}v^{*})+N(Ax^{*}-b)]_{h}))^{+}-v=0.
\end{equation}
We define $U_{1}=\{x\in R^{n}|-\infty\leq x\leq+\infty\},U_{2}=\{x\in R^{h}|0\leq x\leq+\infty\}$, and from $(10)$, we can deduce $NA=M$, so we have\\
$
\begin{array}{l}
x=P_{U_{1}}\{x-\alpha[(\mathrm{I}-M)(Qx+c+B^{\mathrm{T}}v)+N(Ax-b)]\}\\

\quad=P_{U_{1}}\{x-\alpha[(\mathrm{I}-M)(Qx+c+B^{\mathrm{T}}v)+Mx-Nb)]\}\\
\quad=P_{U_{1}}\{ ( \begin{array}{cc} I_{n} & 0_{n\times h}\end{array} )
\left( \begin{array}{c} x\\v \end{array}  \right)\\
\quad\quad-\alpha[((I_{N}-M_{n\times n})Q_{n\times n}+M_{n\times n}(I_{N}-M_{n\times n})B_{n\times h}^{T})
\left( \begin{array}{c} x\\v \end{array}  \right)\\
\quad\quad+(I_{N}-M_{n\times n})c-N_{n\times h}b ] \}\\

v=P_{U_{2}}\{v+\alpha((Bx-d)-\gamma [(\mathrm{I}-M)(Qx+c+B^{\mathrm{T}}v)+N(Ax-b)]_{h})\}\\

\quad=P_{U_{2}}\{ ( \begin{array}{cc} 0_{n\times h} & I_{n}\end{array} )
\left( \begin{array}{c} x\\v \end{array}  \right)-\alpha[(\gamma [(I_{N}-M_{n\times n})Q_{n\times n}+M_{n\times n}]_{h}\\
\quad\quad-B_{n\times h}\gamma [(I_{N}-M_{n\times n})B_{n\times h}^{T}]_{h})
\left( \begin{array}{c} x\\v \end{array}  \right)+d+\gamma [(I_{n}-M_{n\times n})c-N_{n\times h}b]_{h}] \}

\end{array}
$\\
Denote $U=\{x\in R^{n+h}|l\leq x\leq j\},l=\left( \begin{array}{c} -\infty_{n\times 1}\\0_{h\times 1} \end{array} \right),j=\left( \begin{array}{c} +\infty_{n\times 1}\\+\infty_{h\times 1} \end{array}  \right)$,
that is\\
$
\begin{array}{l}
\left( \begin{array}{c} x\\v \end{array}  \right)=P_{U}\{\left( \begin{array}{c} x\\v \end{array}  \right)-\alpha\\
\quad\quad\quad\quad\quad [\left( \begin{array}{cc} (I_{N}-M_{n\times n})Q_{n\times n}+M_{n\times n} & (I_{N}-M_{n\times n})B_{n\times h}^{T}\\\gamma [(I_{N}-M_{n\times n})Q_{n\times n}+M_{n\times n}]_{h}-B_{n\times h} & \gamma [(I_{N}-M_{n\times n})B_{n\times h}^{T}]_{h}  \end{array}  \right)
\\
\quad\quad\quad\quad\quad\left( \begin{array}{c} x\\v \end{array}  \right)+
\left( \begin{array}{c} (I_{n}-M_{n\times n})c-N_{n\times h}b\\
d+\gamma [(I_{n}-M_{n\times n})c-N_{n\times h}b]_{h} \end{array}  \right)]\}
\end{array}$\\
define\\
$$
\begin{array}{l}
y=\left( \begin{array}{c} x\\v \end{array}  \right),
p=\left( \begin{array}{c} (I_{n}-M_{n\times n})c-N_{n\times h}b\\
d+\gamma [(I_{n}-M_{n\times n})c-N_{n\times h}b]_{h} \end{array}  \right),\\
W=\left( \begin{array}{cc} (I_{N}-M_{n\times n})Q_{n\times n}+M_{n\times n} & (I_{N}-M_{n\times n})B_{n\times h}^{T}\\\gamma [(I_{N}-M_{n\times n})Q_{n\times n}+M_{n\times n}]_{h}-B_{n\times h} & \gamma [(I_{N}-M_{n\times n})B_{n\times h}^{T}]_{h}  \end{array}  \right),\\
\end{array}
$$\\
we finally get
\begin{equation}\label{14}
y=P_{U}(y-\alpha(Wy+p))
\end{equation}
After the above analysis, we get the time-varying neural network model to solve$(1)$
\begin{equation}\label{15}
\begin{array}{l}
\left\{
        \begin{array}{l}
\frac{dy}{dt}=-\kappa y(t)+(\kappa-1)P_{U}(y(t-\tau(t))-\alpha(Wy(t-\tau(t))+p)\\
\quad\quad\quad+P_{U}(y(t)-\alpha(Wy(t)+p))
\\
y(t)=\varphi(t),t\in[-\tau,0]
        \end{array}
\right.
\end{array}
\end{equation}
Where $\kappa>0$ is a scale parameter, $P_{U}:R^{n+h}\rightarrow U$ is the projection operator in the sense of Hilbert space, defined by $P_{U}(s)=\arg \min \limits_{z\in U} \parallel s-z \parallel, \forall y\in R^{n+h}$
, where $\| \|$ represents the Euclidean norm, $\tau>0$ denotes the transmission delay. $\alpha p$ is the network input item, $y$ as the network output item, $I-\alpha W$ is connected to weight. If we use  $\Omega ^{\aleph}$ to represent the set of equilibrium points of  $(15)$ and $\Omega ^{*}$ to represent the set of optimal solutions of  $(1)$. Then we will get that if $x^{*}\in \Omega ^{\aleph}$, then there is a $v^{*}$ such that $y^{*}=(x^{*T},v^{*T})^{T}$ satisfies the projection equation $(14)$, which means that $x^{*}\in \Omega ^{*}$. So we have $\Omega ^{\aleph}=\Omega ^{*}$.\\
We Give the following lemma and definitions in preparation for the following discussion.
\begin{lemma}\cite{book23}
If there is a solution $y(t)$ for $(15)$ that satisfies the initial condition $y(t)=\varphi(t), \forall\varphi(t)\in C([-\tau,0],R^{n})t\in[-\tau,0]$, and the solution $y(t)$ is bounded on $[0,T]$, then the existence interval of $y(t)$ is $[0,\infty]$.
\end{lemma}
\begin{definition}\cite{article26}
If $\exists\rho>0, \eta>0$ such that the following inequality holds\\
$$\|y(t)-y^{*}\|\leq \rho\|\varphi-y^{*}\|e^{-\eta t},\forall t\geq0$$, \\
where $\|\varphi-y^{*}\|=\sup \limits_{-\tau \leq t \leq0}[(\varphi(t)-y^{*})^{T}(\varphi(t)-y^{*})]^{\frac{1}{2}}$, then the equilibrium point $y^{*}$ of the time-varying Delay Neural Network defined by $(15)$ is globally exponentially stable.
\end{definition}

\section{Existence and uniqueness}
\begin{theorem}
For $\forall\varphi\in C([-\tau(t),0],R^{n+p})$, the solution of the neural network $(15)$ exists and is unique, $t \in [0,+\infty]$.
\end{theorem}
\begin{proof}
Let\\
 $
 \begin{array}{l}
 Y(y(t))=-\kappa y(t)+(\kappa-1)P_{U}(y(t-\tau(t))-\alpha(Wy(t-\tau(t))+p)+P_{U}(y(t)\\
 \quad\quad\quad\quad\quad-\alpha(Wy(t)+p))
 \end{array}
 $,\\
thus $\frac{dy}{dt}=Y(y(t))$.\\
If $y^{*}$ is used to represent the equilibrium point of the time-varying delay neural network $(15)$, then we can get that
$$
\begin{array}{l}
\|Y(y(t))\|=\|Y(y(t))-Y(y^{*})\|\\
\quad\quad\quad\quad\leq\kappa\|y(t)-y^{*}\|+(\kappa-1)\|P_{U}(y(t-\tau(t))-\alpha(Wy(t-\tau(t))+p)\\
\quad\quad\quad\quad\quad\quad-P_{U}(y^{*}-\alpha(Wy^{*}+p)))\|+\|P_{U}(y(t)-\alpha(Wy(t)+p)\\
\quad\quad\quad\quad\quad\quad-P_{U}(y^{*}-\alpha(Wy^{*}+p)))\|\\
\quad\quad\quad\quad\leq\kappa\|y(t)-y^{*}\|+(\kappa-1)(\|y(t-\tau(t))-y^{*}\|-\|\alpha W\|\|y(t)-y^{*}\|)\\
\quad\quad\quad\quad\quad\quad+(\|y(t)-y^{*}\|-\|\alpha W\|\|y(t)-y^{*}\|)\\
\quad\quad\quad\quad\leq\kappa(2+\|\alpha W\|)y^{*}+(\kappa+(1+\|\alpha W\|))y(t)\\
\quad\quad\quad\quad\quad\quad+(\kappa-1)(1+\|\alpha W\|)y(t-\tau(t))
\end{array}
$$
Let $\beta_{1}=\kappa(2+\|\alpha W\|),\beta_{2}=\kappa+(1+\|\alpha W\|),\beta_{2}=(\kappa-1)(1+\|\alpha W\|)$,
since
$$ y(x)=\left\{
\begin{array}{l}
y(0)+\int^{t}_{0}Y(y(s))ds,t\in[0,T] \\
\varphi(t),t\in[-\tau,0] \\
\end{array}
\right.,
$$
it can be concluded that\\
$$
\begin{array}{l}
\|y(t)\|
\leq\|\varphi(t)\|+\int^{t}_{0}\|Y(y(s))\|ds\\
\quad\quad\quad\leq\|\varphi(t)\|+\int^{t}_{0}\beta_{1}\|y^{*}\|+\beta_{2}\|y(s)\|+\beta_{3}\|y(s-\tau(s)\|ds\\
\quad\quad\quad\leq\|\varphi(t)\|+\beta_{1}\|y^{*}\|T+\beta_{2}\int^{t}_{0}\|y(s)\|ds+\beta_{3}\int^{t-\tau(s)}_{-\tau(s)}\|y(s)\|ds\\
\quad\quad\quad\leq(1+\beta_{3}\tau)\|\varphi(t)\|+\beta_{1}\|y^{*}\|T+(\beta_{2}+\beta_{3})\int^{t}_{0}\|y(s)\|ds
\end{array}
$$
From Bellman's inequality, we have
$$\|y(t)\|\leq[(1+\beta_{3}\tau)\|\varphi(t)\|+\beta_{1}\|y^{*}\|T]e^{(\beta_{2}+\beta_{3})t},
t\in[0,T].$$
That is, the boundedness of $\|y(t)\|$ on $[0,T]$ has been proved.
By Lemma 1, $\exists y(t)$ for $(14)$ on $[0,\infty]$.
A discussion on the uniqueness of $y(t)$ is given below.
Suppose the solution is not unique, then there is a solution with $y(t),\tilde{y(t)},(y(t)\neq\tilde{y(t)})$ being $(15)$. Express the solutions of $y(t),\tilde{y(t)}$ through the variation-of-constants formula
\begin{equation}
\begin{array}{l}
y(t)=e^{-\kappa t}y(0)+(\kappa-1)\int^{t}_{0}e^{-\kappa(t-s)}P_{U}(y(s-\tau(s))-\alpha(Wy(s-\tau(s))+p))ds
\\ \quad\quad\quad+\int^{t}_{0}e^{-\kappa(t-s)}P_{U}(y(s)-\alpha(Wy(s)+p))ds
\end{array}
\end{equation}
\begin{equation}
\begin{array}{l}
\tilde{y(t)}=e^{-\kappa t}y(0)+(\kappa-1)\int^{t}_{0}e^{-\kappa(t-s)}P_{U}(\tilde{y(s-\tau(s))}-\alpha(W\tilde{y(s-\tau(s))}+p))ds
\\ \quad\quad\quad+\int^{t}_{0}e^{-\kappa(t-s)}P_{U}(\tilde{y(s)}-\alpha(W\tilde{y(s)}+p))ds
\end{array}
\end{equation}
Subtract $(16)$ and $(17)$, we have

$$
\begin{array}{l}
\sup\limits_{t}\|y(t)-\tilde{y(t)}\| \leq \sup\limits_{t} \|(\kappa-1)\int^{t}_{0}e^{-\kappa(t-s)}P_{U}(y(s-\tau(s))-\tilde{y(s-\tau(s))}\\
\quad\quad\quad\quad\quad\quad\quad\quad\quad-\alpha W(y(s-\tau(s))-\tilde{y(s-\tau(s))}))ds\|\\
\quad\quad\quad\quad\quad\quad\quad\quad\leq \sup\limits_{t}|\kappa-1|\int^{t}_{0}e^{-\kappa(t-s)}\|I-\alpha W\| \|y(s-\tau(s))-\tilde{y(s-\tau(s))}\|ds\\
\quad\quad\quad\quad\quad\quad\quad\quad\quad+\sup\limits_{t}\int^{t}_{0}e^{-\kappa(t-s)}\|I-\alpha W\| \|y(s)-\tilde{y(s)}\|ds\\
\quad\quad\quad\quad\quad\quad\quad\quad\leq |\kappa-1|\|I-\alpha W\| \sup\limits_{t}\|y(t-\tau(t))-\tilde{y(t-\tau(t))}\|\int^{t}_{0}e^{-\kappa(t-s)}ds\\
\quad\quad\quad\quad\quad\quad\quad\quad\quad+\|I-\alpha W\|\sup\limits_{t}\|y(t)-\tilde{y(t)}\|\int^{t}_{0}e^{-\kappa(t-s)}ds\\
\quad\quad\quad\quad\quad\quad\quad\quad\leq |\kappa-1|\|I-\alpha W\| \sup\limits_{t}\|y(t)-\tilde{y(t)}\|\int^{t}_{0}e^{-\kappa(t-s)}ds\\
\quad\quad\quad\quad\quad\quad\quad\quad\quad+\|I-\alpha W\|\sup\limits_{t}\|y(t)-\tilde{y(t)}\|\int^{t}_{0}e^{-\kappa(t-s)}ds\\
\quad\quad\quad\quad\quad\quad\quad\quad\leq \frac{|\kappa-1|+1}{\kappa}\|I-\alpha W\|\sup\limits_{t}(1-e^{-\kappa t})\|y(t)-\tilde{y(t)}\|
\end{array}
$$
which implies that
\begin{equation}
(1-\frac{|\kappa-1|+1}{\kappa}\|I-\alpha W\|\sup\limits_{t}(1-e^{-\kappa t}))\sup\limits_{t}\|y(t)-\tilde{y(t)}\|\leq 0.
\end{equation}
Yet
\begin{equation}
1-\frac{|\kappa-1|+1}{\kappa}\|I-\alpha W\|\sup\limits_{t}(1-e^{-\kappa t})>0.
\end{equation}
According to $(18)$ and $(19)$, we have
$$\sup\limits_{t}\|y(t)-\tilde{y(t)}\|\leq 0.$$
Therefore, $y(t)=\tilde{y(t)}$, this contradicts the above assumption. That is to say, the $(15)$ has a unique solution.
\end{proof}

\section{Global exponetial stability}

\begin{theorem}
When condition $(|\kappa-1|+1)\|I-\alpha W\|-\kappa<0$, the equilibrium point $y^{*}$ of neural networks with time-varying delays has global exponential stability.
\end{theorem}

\begin{proof}
Since $y^{*}$ is the equilibrium point of $(15)$, the following equation can be obtained
$$
\begin{array}{l}
\frac{d(y(t)-y^{*})}{dt}=-\kappa(y(t)-y^{*})+(\kappa-1)P_{U}((y(t-\tau(t))-y^{*})\\
\quad\quad\quad\quad\quad\quad-\alpha(W(y(t-\tau(t))-y^{*})+p))\\
\quad\quad\quad\quad\quad\quad+P_{U}((y(t)-y^{*})-\alpha(W(y(t)-y^{*})+p))
\end{array}
$$
From the above formula, $y(t)-y^{*}$ can be expressed as the following form through the variation-of-constants
$$
\begin{array}{l}
y(t)-y^{*}=e^{-\kappa t}(\varphi-y^{*})+(\kappa-1)\int^{t}_{0}P_{U}((y(s-\tau(s))-y^{*})\\
\quad\quad\quad\quad\quad\quad-\alpha(W(y(s-\tau(s))-y^{*})+p))e^{-\kappa(t-s)}ds\\
\quad\quad\quad\quad\quad\quad+\int^{t}_{0}P_{U}((y(s)-y^{*})-\alpha(W(y(s)-y^{*})+p))e^{-\kappa (t-s)}ds
\end{array}
$$
then
$$
\begin{array}{l}
\|y(t)-y^{*}\|\leq e^{-\kappa t}\|\varphi-y^{*}\|+|\kappa-1|\int^{t}_{0}\|(y(s-\tau(s))-y^{*})\\
\quad\quad\quad\quad\quad\quad-\alpha(W(y(s-\tau(s))-y^{*})+p)\|e^{-\kappa(t-s)}ds\\
\quad\quad\quad\quad\quad\quad+\int^{t}_{0}\|(y(s)-y^{*})-\alpha(W(y(s)-y^{*})+p)\|e^{-\kappa (t-s)}ds\\
\quad\quad\quad\quad\quad=e^{-\kappa t}\|\varphi-y^{*}\|+|\kappa-1|\int^{t}_{-\tau(s)}\|(I-\alpha W)(y(s)-y^{*})-\alpha p\|\\
\quad\quad\quad\quad\quad\quad e^{-\kappa (t-s-\tau(s))}ds+\int^{t}_{0}\|(I-\alpha W)(y(s)-y^{*})-\alpha p\|e^{-\kappa(t-s)}ds\\
\quad\quad\quad\quad\quad\leq e^{-\kappa t}\|\varphi-y^{*}\|+\frac{|\kappa-1|}{\kappa}(\|I-\alpha W\| \|\varphi-y^{*}\|e^{-\kappa t})-\frac{|\kappa-1|}{\kappa}\alpha pe^{-\kappa t}\\
\quad\quad\quad\quad\quad\quad+(|\kappa-1|+1)\int^{t}_{0}\|(I-\alpha W)(y(s)-y^{*})-\alpha p\|e^{-\kappa(t-s)}ds\\
\quad\quad\quad\quad\quad=e^{-\kappa t}(\frac{(\kappa+|\kappa-1|\|I-\alpha W)\|\varphi-y^{*}\|+|\kappa-1|\alpha p}{\kappa})\\
\quad\quad\quad\quad\quad\quad+(|\kappa-1|+1)\int^{t}_{0}\|(I-\alpha W)(y(s)-y^{*})-\alpha p\|e^{-\kappa(t-s)}ds
\end{array}
$$
 As in the following inequality, $e^{\kappa t}$ is moved to the left-hand side of the inequality
 $$
\begin{array}{l}
\|y(t)-y^{*}\|e^{\kappa t} \leq \frac{(\kappa+|\kappa-1|\|I-\alpha W)\|\varphi-y^{*}\|+|\kappa-1|\alpha p}{\kappa}\\
\quad\quad\quad\quad\quad\quad\quad\quad+(|\kappa-1|+1)\int^{t}_{0}\|(I-\alpha W)(y(s)-y^{*})-\alpha p\|e^{\kappa s}ds\\
\quad\quad\quad\quad\quad\quad\quad\leq \frac{(\kappa+|\kappa-1|\|I-\alpha W)\|\varphi-y^{*}\|+|\kappa-1|\alpha p}{\kappa}e^{(|\kappa-1|+1)\|I-\alpha W)\|t}
\end{array}
$$
That is
$$
\begin{array}{l}
\|y(t)-y^{*}\| \leq \frac{(\kappa+|\kappa-1|\|I-\alpha W)\|\varphi-y^{*}\|+|\kappa-1|\alpha p}{\kappa}e^{[(|\kappa-1|+1)\|I-\alpha W)\|-\kappa]t}
\end{array}
$$
So we can obtained that if $(|\kappa-1|+1)\|I-\alpha W)\|-\kappa<0$, the time-varying delay neural network defined by $(15)$ is globally exponentially stable.
\end{proof}
\section{Simulation results}
\begin{example}
Consider the following quadratic programming
$$
\begin{array}{l}
\min f(x)=0.36x^{2}_{1}+0.3x^{2}_{2}+0.2x^{2}_{3}-x_{1}+0.6x_{2}+0.5x_{3}\\
Subject\quad to
\left\{
        \begin{array}{l}
x_{1}-x_{2}+x_{3}=6\\
             0.5x_{1}-0.7x_{2}+0.2x_{3}\leq 5\\
             \frac{1}{4}x_{1}+ \frac{2}{5}x_{2}- \frac{3}{5}x_{3}\leq 7
        \end{array}
\right.
\end{array}
$$
Let $Q=\left( \begin{array}{ccc}
0.72 & 0 & 0\\0 & 0.6 & 0\\0 & 0 & 0.4
\end{array}  \right)$,
$c=\left( \begin{array}{c}
-1\\ 0.6\\ 0.5
\end{array}  \right)$,
$A=\left( \begin{array}{ccc}
1 & -1 & 1
\end{array}  \right)$,$b=6$,
$B=\left( \begin{array}{ccc}
0.5 & -0.7 & 0.2\\ \frac{1}{4} & \frac{2}{5} & \frac{3}{5}
\end{array}  \right)$,
$d=\left( \begin{array}{c}
5 \\ 7
\end{array}  \right)$. The eigenvalues of $Q$ can be calculated as $\lambda_{1}=0.4000$, $\lambda_{2}=0.6000$, $\lambda_{3}=0.7200$. The optimal solution for this example can be calculated to be $x^{*}=\left( \begin{array}{ccc}
0.4000 & 0.6000 & 0.7200
\end{array}  \right)^{T}$.
Next we choose $u=0.003$, $\alpha=0.45$, $\gamma=1$, $\kappa=2$, and calculated the
$$
\begin{array}{c}
M=\left( \begin{array}{ccc}
0.3333 & -0.3333 & 0.3333\\-0.3333 & 0.3333 & -0.3333\\0.3333 & -0.3333 & 0.3333
\end{array}  \right),
N=\left( \begin{array}{c}
0.3333 \\ -0.3333 \\ 0.3333
\end{array}  \right),\\
W=\left( \begin{array}{ccccc}
0.8133 & -0.1333 & 0.2000 & 0.0333 & 0.5556\\
-0.0933 & 0.7333 & -0.2000 & -0.2333 & 0.1778\\
0.0933 & -0.1333 & 0.6000 & 0.2667 & -0.3778\\
0.3133 & 1.1667 & 0.4000 & -0.4667 & 0.3556\\
0.4800 & 0.0667 & 1.2000 & -0.4667 & 0.3556
\end{array}  \right)
\end{array}
$$,
$(|\kappa-1|+1)\|I-\alpha W)\|-\alpha<0$,
through Theorem 4.1, we can know that the equilibrium point of the time-delay neural network $(15)$ is globally exponentially stable. We obtain the state trajectory of the time-delay neural network $(15)$ corresponding to Example 1 through Matlab2016a. The trajectory corresponds to 10 sets of random initial functions, and $\tau=0.365$. From Figure1\ref{fig:example1}, we can see that the state trajectory of the neural network globally converges to the optimal solution of the quadratic programming in Example 1.
 \begin{figure}
  \centering
  \includegraphics[width=5in]{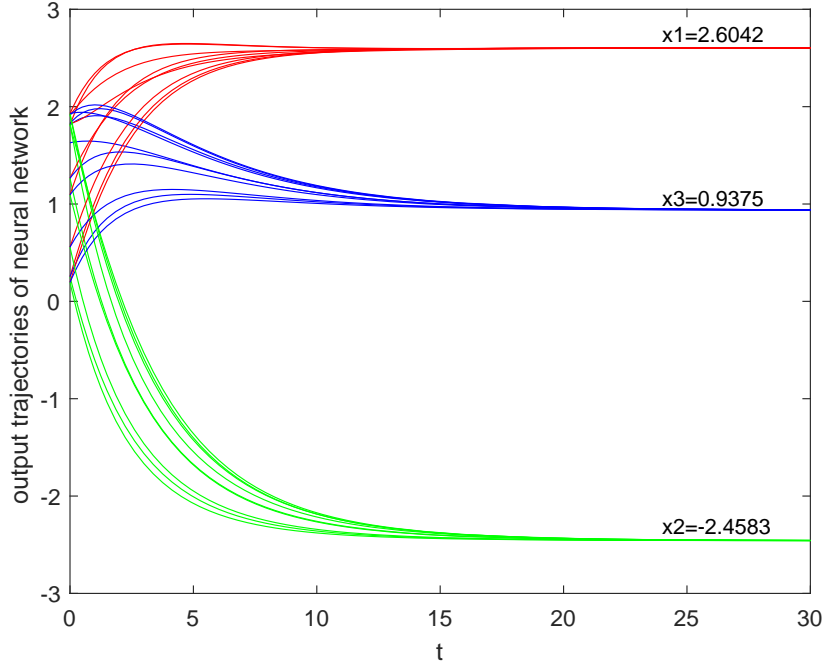}
  \caption{The state trajectory of the time-delay neural network corresponding to Example 1}\label{fig:example1}
\end{figure}

\end{example}

\begin{example}
Consider the following quadratic programming
$$
\begin{array}{l}
\min f(x)=0.7x^{2}_{1}+0.6x^{2}_{2}+0.25x^{2}_{3}+2x^{2}_{4}+0.35x_{1}x_{2}+0.45x_{1}x_{3}+0.25x_{2}x_{3}\\
\quad\quad\quad+\frac{1}{9}x_{2}x_{4}+0.36x_{1}+0.79x_{2}-9x_{3}-8x_{4}\\
Subject\quad to
\left\{
        \begin{array}{l}
x_{1}+0.5x_{2}-x_{3}-0.95x_{4}=4\\
             x_{1}+0.2x_{2}-0.3x_{3}+0.6x_{4}\leq 4.5\\
             -0.6x_{1}+ x_{2}+0.13x_{3}-0.3x_{4}\leq 3.5
        \end{array}
\right.
\end{array}
$$
Let
$$
\begin{array}{c}
Q=\left( \begin{array}{cccc}
1.4 & 0.35 & 0.45 & 0\\0.35 & 1.2 & 0.25 & \frac{1}{9}\\
0.45 & 0.25 & 0.5 & 0\\0 & \frac{1}{9} & 0 & 4
\end{array}  \right),
c=\left( \begin{array}{c}
0.36\\ 0.79\\ -9\\ -8
\end{array}  \right),
B=\left( \begin{array}{cccc}
1 & 0.2 & -0.3 & 0.6\\ -0.6 & 1 & 0.13 & -0.3
\end{array}  \right),\\
A=\left( \begin{array}{cccc}
1 & 0.5 & -1 & -0.95
\end{array}  \right),
b=4,
d=\left( \begin{array}{c}
4.5 \\ 3.5
\end{array}  \right)
\end{array}
$$. The eigenvalues of $Q$ can be calculated as $\lambda_{1}=0.3012$, $\lambda_{2}=0.9396$, $\lambda_{3}=1.8547$, $\lambda_{4}=4.0045$. The optimal solution for this example can be calculated to be $x^{*}=\left( \begin{array}{cccc}
2.6080 & 1.8757 & -0.5792 & 0.1317
\end{array}  \right)^{T}$.
Next we choose $u=0.001$, $\alpha=0.75$, $\gamma=1$, $\kappa=2$, and calculated the
$$
\begin{array}{c}
M=\left( \begin{array}{cccc}
0.3172 & 0.1586 & -0.3172 & -0.3013\\
0.1586 & 0.0793 & -0.1586 & -0.1507\\
-0.3172 & -0.1586 & 0.3172 & 0.3013\\
-0.3013 & -0.1507 & 0.3013 & 0.2863
\end{array}  \right),
W=\left( \begin{array}{c}
0.3172 \\ 0.1586 \\ -0.3172 \\-0.3013
\end{array}  \right),\\
W=\left( \begin{array}{cccccc}
1.3603 & 0.3200 & 0.1090 & 0.8864 & 0.7367 &-0.6174\\
0.3302 & 1.1850 & 0.0795 & 0.5543 & 0.0684 &0.9913\\
0.4897 & 0.2800 & 0.8410 & -0.8864 & -0.0376 &0.1474\\
0.0377 & 0.1396 & 0.3239 & 3.1579 & 0.8501 &-0.2834\\
1.2178 & 1.7246 & 1.6534 & 3.1122 & 1.6185 &0.2379\\
2.2178 & 0.9246 & 1.2234 & 4.0122 & 1.6185 &0.2379
\end{array}  \right),
\end{array}
$$
$(|\kappa-1|+1)\|I-\alpha W)\|-\alpha<0$,
through the Theorem 4.1, we would know that the equilibrium point of the time-delay neural network $(15)$ is globally exponentially stable. We obtain the state trajectory of the time-delay neural network $(15)$ corresponding to Example 2 through Matlab2016a. The trajectory corresponds to 20 sets of random initial functions, and $\tau=0.03$. From the Figure2\ref{fig:example2}, we can see that the state trajectory of the neural network globally converges to the optimal solution of the quadratic programming in Example 2.
 \begin{figure}
  \centering
  \includegraphics[width=5in]{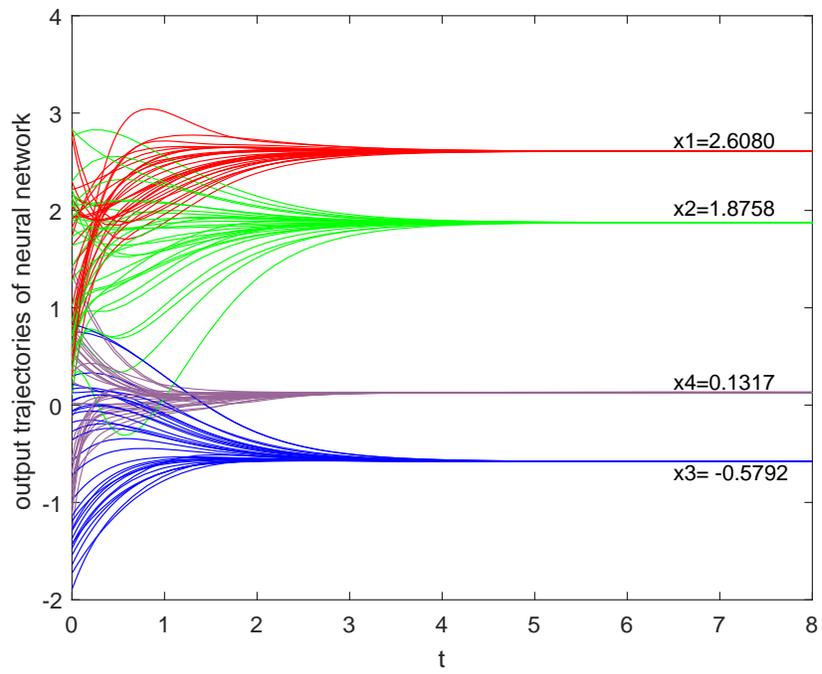}
  \caption{The state trajectory of the time-delay neural network corresponding to Example 2}\label{fig:example2}
\end{figure}

\end{example}

\section{Conclusion}
\par \setlength{\parindent}{2em}This paper proposes a class of neural network models with variable time delays to solve convex optimization problems. Compared with constant time delays, the discussion of variable time delay has better practical value. The equilibrium point of the neural network corresponds to the optimal solution of the convex optimization problem. Therefore, it is meaningful to use the neural network with $n+h$ neurons to solve the optimization problem in practice. For the proposed neural network, it is proved that the equilibrium point of the neural network exists and is unique, we discussed that it is globally exponentially stable under certain conditions. Some examples are given to illustrate the practicability of the network.

\section*{References}

\bibliography{finalbib}
\begin{CJK}{GBK}{kai}
\renewcommand\refname{references}

\end{CJK}

\end{document}